\documentclass{article}

\usepackage{amsmath,amssymb,amsthm,amsfonts}
\usepackage{url}
\theoremstyle{plain}
\newtheorem{thm}{Theorem}[section]

\theoremstyle{definition}

\begin{document}

\title{Approximations of  additive squares in infinite words}

\author{Tom Brown}
\date{July 2011}
\maketitle

\begin{abstract}  

We show that every infinite word $\omega$ on a finite subset of $\mathbb{Z}$ must contain arbitrarily large factors $B_1B_2$ which are ``close" to being \textit{additive squares}.  We also show that for all $k>1, \ \omega$ must contain a factor  $U_1U_2 \cdots U_k$ where $U_1,U_2, \cdots, U_k$ all have the same \textit{average.}

\end{abstract}

\section{Introduction}

If $S$ is a finite subset of $\mathbb{Z}$ and $\omega \in S^{\mathbb{N}}$, we write $\omega = x_1x_2x_3 \cdots$.  For any (finite) factor $B=x_ix_{i+1} \cdots x_{i+n}$ of $\omega$, we write $|B|$ for the \textit{length} of $B$ (here $|B|= n+1$), and we write $$\sum B = x_i+x_{i+1}+ \cdots +x_{i+n}.$$

If $B_1B_2$ is a factor of $\omega$ with $$|B_1|=|B_2| \ \ \text{and}  \ \sum B_1 = \sum B_2,$$  we say that $B_1B_2$ is an \textit{additive square} contained in $\omega$.  For example, if $\omega = 2 1 3 5 1 2 6 \cdots$ (a word on the alphabet $S=\{1,2,3,4,5,6\}$), then $\omega$ contains the additive square $B_1B_2$, where $B_1 = 1 3 5, B_2 = 1 2 6,$ with $|B_1|=|B_2|=3$ and  $\sum B_1 = \sum B_2=9.$

A celebrated result of V. Ker\"{a}nen \cite{Keranen1} (see also \cite{Keranen2}) is that there exist infinite words $\omega$ on an alphabet of 4 symbols which contain no \textit{abelian square}, that is, $\omega$ contains no factor $B_1B_2$ where $B_1,B_2$ are permutations of one another.  (For early background, see \cite{Brown}.)

After Ker\"{a}nen's result, it was natural to consider the question of whether an infinite word $\omega$ on 4 (or more) integers must contain an \textit{additive square}.  

Allen Freedman \cite{Freedman} showed that if $a,b,c,d \in \mathbb{Z}$ (or more generally if $a,b,c,d$ belong to any field of characterisitic 0) and $a+d=b+c,$ then every word of length 61 on $\{a,b,c,d\}$ contains an additive square.  

Julien Cassaigne, James D. Currie, Luke Schaeffer, and Jeffrey Shallit \cite{CCSS} showed that there is an infinite word $\omega$ on the alphabet $\{0,1,3,4\}$ which contains no \textit{additive cube}, that is, $\omega$ contains no factor $B_1B_2B_3$ such that $|B_1|=|B_2|=|B_3|$ and $\sum B_1 = \sum B_2 = \sum B_3$.

A few remarks follow.

For each $k \ge 1,$ let $g(1,2,\cdots,k)$ denote the length of a longest word on $\{1,2,$ $\cdots,k\}$ which does not contain an additive square. (We allow $g(1,2,\cdots,k) = \infty$.)  Then the following three statements are equivalent:

1. For all $k \ge 1, \ \ g(1,2,\cdots,k) < \infty.$

2. For all  $k \ge 1,$ and all infinite words $\omega$ on $\{1,2,\cdots,k\}, \ \omega$ contains arbitrarily large additive squares.

3. Let $x_1 < x_2 < x_3 \cdots $ be any sequence of positive integers such that, for some $M$, $0<x_{i+1}-x_i <M$ for all $i \ge 1.$  Then there exist $i < j < k$ such that both $\{i,j,k\}$ and $\{x_i,x_j,x_k\}$ are arithmetic progressions.
(Statement 3 is equivalent to statement 1 via van der Waerden's theorem on arithmetic progressions \cite{vdW}.)

Finally, let us denote by $g(a,b,c,d)$ the length of a longest word on $\{a,b,c,d\}$ which does \textit{not} contain an additive square.  Then the statement $$\lim_{n \to \infty}g(1,n,n^2,n^3)=\infty$$ is equivalent (by standard combinatorial arguments) to the result of Ker\"{a}nen stated above.

The question concerning the presence of additive squares seems to have appeared in print for the first time in a paper by Giuseppe Pirillo and Stefano Varricchio \cite{Pirillo}.  Other related material can be found in \cite{abjs}, \cite{Au}, \cite{CCSS}, \cite{CRSZ},  \cite{Freedman}, \cite{Jaroslaw}, \cite{HH}, \cite{Pirillo}, and \cite{RSZ}.

In this note we show that for every finite subset $S$ of $\mathbb{Z}$ there is a constant $C$ (which depends only on $S$) such that every infinite word $\omega$ on $S$ contains arbitrarily long factors $UV$ such that $$|U|= |V| \ \ \text{and} \ \ |\sum U- \sum V| \le C.$$

We also show that for every infinite word $\omega$ on a finite subset of $\mathbb{Z}$ there must exist, for every $k > 1,$ a factor $B_1B_2 \cdots B_k$ of $\omega$ such that $B_1,B_2, \cdots, B_k$  all have the same \textit{average}.  Here, the \textit{average} of a factor $B$ is $\frac{1}{|B|}\sum B.$\


\section{Adjacent equal length blocks with nearly equal sums}

Here we exploit the fact that if $U,V$ are words on a 2-element subset of $\mathbb{Z}$, then $UV$ is an \textit{additive} square ($|U|=|V|$ and $\sum U = \sum V$) if and only if $UV$ is an \textit{abelian} square  ($U$ and $V$ are permutations of one another).

\begin {thm}

For every finite subset $S$ of $\ \mathbb{Z}$ there exists a constant $C$ (depending only on $S$) such that every infinite word $\omega$ on $S$ contains arbitrarily long factors $UV$ such that $$|U|= |V| \ \ \text{and} \ \ |\sum U- \sum V| \le C.$$

\end{thm}

\begin{proof}

First assume that $S$ is a finite subset of $\mathbb{N}$, and let $\omega = x_1x_2x_3 \cdots$ be an infinite word on $S$. Let $1^{x_i}$ denote a string of 1s of length $x_i$ (e.g., $1^4 = 1111$), and let $\omega^*$ be the binary word $1^{x_1}01^{x_2}01^{x_3}0 \cdots $, which we write for convenience as $x_10x_20x_30 \cdots $.  By a theorem of Entringer, Jackson, and Schatz \cite {Entringer} the word $\omega^*$ contains arbitrarily large abelian squares $UV$, and hence arbitrarily large factors $UV$ with $|U|=|V|$ and $\sum U = \sum V$.  Re-numbering the indices for convenience, such a square $UV$, since each of $U$ and $V$ must contain the same number, say $k$, of 0s, has the form $$U = \alpha_2 0 x_2 0 x_3 0 \cdots 0 x_k 0  \alpha_3, V = \alpha_4  x_{k+2} 0 \cdots 0 x_{2k} 0 \alpha_5,$$ where $\alpha_1 + \alpha_2 = x_1, \alpha_3 + \alpha_4 = x_{k+1}, \alpha_5 + \alpha_6 = x_{2k+1}.$  (All the $\alpha_i$ are non-negative integers.)  Since $UV$ is an additive square, 

$$\alpha_2 +\sum _{i=2}^kx_i + \alpha_3  = \alpha_4 +\sum _{i=k+2}^{2k}x_i + \alpha_5,$$ or (using $\alpha_1 + \alpha_2 = x_1$ and $ \alpha_3 + \alpha_4 = x_{k+1}$)

$$|\sum _{i=1}^kx_i - \sum _{i=k+1}^{2k+1}x_i| = |\alpha_1 -2\alpha_3 + \alpha_5| \le 2\max S,$$ hence we have $$|U|=|V| \ \ \text{and} \ \ 
|\sum U-\sum V|\le 2\max S.$$

When  $S$  is a finite subset of $\mathbb{Z}$ which contains non-positive integers, translate $S$ to the right by $|\min S|+1$ and apply the above argument, to get arbitrarily large factors $UV$ such that $$|U| = |V| \ \ \text{and}\ \ |\sum U - \sum V| \le 2(|\min S|+ 1+ \max S).$$

\end{proof}


\section{Adjacent factors with equal averages}


\begin{thm}  

For any finite subset $S$ of $\ \mathbb{Z}$, any infinite word $\omega$ on $S$, and any $k > 1$, there exists a factor $U_1U_2 \cdots U_k$ with $$\frac{1}{|U_1|}\sum U_1 = \frac{1}{|U_2|}\sum U_2 = \cdots = \frac{1}{|U_k|}\sum U_k.$$ 

\end{thm}

\begin {proof}

Let $\omega = x_1x_2x_3 \cdots$ be a given infinite word on the set of integers $S=\{s_1, s_2,\cdots, s_t\}.$  Consider the infinite sequence of points in the plane $P_i = (i, x_1 + x_2 + \cdots + x_i), i\ge 1.$  Since $P_{i+1} - P_i = (1, x_{i+1}) \in \{(1, s_j): 1 \le j \le t\}$, a theorem of Gerver and Ramsey \cite{Gerver}  asserts that the set $\{P_i: \ i \ge 1\}$ contains, for any given $k > 1, \ k+1$ collinear points $P_{i_1}  P_{i_2}  \cdots P_{i_{k+1}}$.  For $1 \le j \le k,$ let $U_j=x_{{i_j+1}}x_{{i_j+2}}\cdots x_{i_{j+1}}.$ The slope of the line segment joining $P_{i_{j}}$ and $P_{i_{j+1}}$ is $\frac{1}{|U_j|}\sum U_j.$  Since this slope is the same for each choice of $j$, we have $$\frac{1}{|U_1|}\sum U_1 = \frac{1}{|U_2|}\sum U_2 = \cdots = \frac{1}{|U_k|}\sum U_k.$$ 
\end{proof}


\paragraph{Acknowledgments.} The author would like to thank Allen Freedman, Veso Jungi\'{c}, Hayri Ardal, and Julian Sahasrabudhe for helpful conversations, and to acknowledge the IRMACS Centre at Simon Fraser University for its support.

\bigskip

\noindent\textit{Department of Mathematics,
Simon Fraser University, Burnaby, BC, Canada, V5A 1S6\\
tbrown@sfu.ca\\}

\end{document}